\documentclass[11pt,reqno]{amsart}
\usepackage{amssymb,amsmath}\usepackage{subfigure}\newcommand{\be}{\begin{eqnarray}}
\newcommand{\ee}{\end{eqnarray}}
\newcommand{\beno}{\begin{eqnarray*}}
\newcommand{\eeno}{\end{eqnarray*}}
\newcommand{\e}{{\varepsilon}}

\newcommand{\om}{\omega}
\newcommand{\supp}{\operatorname{supp}}\newtheorem{theorem}{Theorem}\newtheorem{lemma}[theorem]{Lemma}\theoremstyle{definition}\theoremstyle{remark}\numberwithin{equation}{section}\input epsf.sty
\begin{document}\thispagestyle{empty}
\newcommand{\diag}{\operatorname{diag}}
%
%
%
%
%
%
%BRAND NEW COMMANDS%
%
%
%
%
%
%
\newcommand{\G}{{\mathcal G}}
\newcommand{\C}{{\mathbb C}}
\newcommand{\Pa}{{P_{1,\theta}(y)}}
\newcommand{\Pb}{{P_{2,\theta}(y)}}
\newcommand{\Pshp}{{P_{1,t}^\sharp(x)}}
\newcommand{\Pflt}{{P_{1,t}^\flat(x)}}
\newcommand{\OTL}{{\Omega(\theta,\ell)}}
\newcommand{\rsz}{{R(\theta^*)}}
\newcommand{\phitil}{{\tilde{\varphi}_t}}
\newcommand{\lambdatil}{{\tilde{\lambda}}}
\newcommand{\al}{\alpha}
\newcommand{\ftpI}{\langle (g^2 + \tau^2 f^2)^{\frac{p}{2}} \rangle_{I}}
\newcommand{\fI}{\langle f \rangle_{I}}
\newcommand{\fpI}{\langle |f|^p \rangle_{I}}
\newcommand{\gI}{\langle g \rangle_{I}}
\newcommand{\mtp}{|(f,h_J)| = |(g,h_J)|}
\newcommand{\fJp}{\langle f \rangle_{J^+}}
\newcommand{\fJm}{\langle f \rangle_{J^-}}
\newcommand{\gJp}{\langle f \rangle_{J^+}}
\newcommand{\gJm}{\langle f \rangle_{J^-}}
\newcommand{\fIpm}{\langle f \rangle_{J^{\pm}}}
\newcommand{\gIpm}{\langle g \rangle_{J^{\pm}}}
\newcommand{\fIp}{\langle f \rangle_{I^{+}}}
\newcommand{\fIm}{\langle f \rangle_{I^{-}}}
\newcommand{\gIp}{\langle g \rangle_{I^{+}}}
\newcommand{\gIm}{\langle g \rangle_{I^{-}}}
\newcommand{\yone}{y_1}
\newcommand{\ytwo}{y_2}
\newcommand{\yonep}{y_1^+}
\newcommand{\ytwop}{y_2^+}
\newcommand{\yonem}{y_1^-}
\newcommand{\ytwom}{y_2^-}
\newcommand{\ythr}{y_3}

\newcommand{\xone}{x_1}
\newcommand{\xtwo}{x_2}
\newcommand{\xonep}{x_1^+}
\newcommand{\xtwop}{x_2^+}
\newcommand{\xonem}{x_1^-}
\newcommand{\xtwom}{x_2^-}
\newcommand{\xthr}{x_3}
\newcommand{\ptwo}{\frac{p}{2}}
\newcommand{\alp}{\al^+}
\newcommand{\alm}{\al^-}
\newcommand{\alpm}{\al^{\pm}}
\newcommand{\twrp}{\frac{2}{p}}
\newcommand{\sbeta}{\sqrt{\om^2 - \tau^2}}
\newcommand{\ssbeta}{\om^2 - \tau^2}
\newcommand{\sign}{\operatorname{sign}}
\newcommand{\rst}[1]{\ensuremath{{\mathbin\upharpoonright}%
\newcommand{\and}{\operatorname{and}}

\raise-.5ex\hbox{$#1$}}}

\title[]{{Astala's conjecture from the point of view of singular integrals on metric spaces}}
%\author{Alexander Volberg}\address{Alexander Volberg, Dept. of  Math., Michigan State University. 
%{\tt volberg@math.msu.edu}}%\,.
\author{Alexander Volberg}
\footnote{Supported by NSF grant DMS 0758552; AMS subject classifications: 30C62, 35J15,  35J70, 42B20, 42B35}
\begin{abstract}
In the proof of Astala's conjecture on quasiconformal distortion obtained by Lacey--Sawyer--Uriarte-Tuero one of the key point
is an estimate of the Ahlfors--Beurling operator in a certain weighted space. We show that the point of view of non-homogeneous Harmonic Analysis simplifies considerably this key point.
\end{abstract}
\maketitle

\section{Introduction}

Let us explain the setting. Let $\phi$ be a $K$-quasiconformal (QC) mapping, $0<t<2$. Astala \cite{A} proved the following celebrated
$$
\dim (E) =t \Rightarrow \dim (\phi(E)) \le t^{'}\,,
$$
where
\begin{equation}
\label{t}
\frac{1}{t^{'}} -\frac12 =\frac1{K} \bigg( \frac1{t}-\frac12\bigg)\,.
\end{equation}

He asked whether for the borderline distortion one should have the absolute continuity of corresponding Hausdorff measures:
\begin{equation}
\label{A}
\mathcal{H}^t (E)=0\Rightarrow \mathcal{H}^{t^{'}}(E) =0\,?
\end{equation}

The answer is ``yes", it was proved in \cite{LSUT}, and one of the key point was the following theorem about weighted estimate of singular integral operators (SIO).

\begin{theorem}
\label{main}
Let $0<d<2$, and $\{Q_m\}_{m=1}^{M}$ is a collection of dyadic squares on the plane such that it satisfy strong {\it disjointness condition}
\begin{equation}
\label{D}
4 Q_m\cap 4 Q_{m'} =\emptyset\,,\,\forall m\neq m'\,,
\end{equation}
and  {\it  a packing condition}, which says that for any dyadic square $Q$ of the same lattice
\begin{equation}
\label{P}
\sum_{Q_m\subset Q} \ell(Q_m)^{2-d} \le C\, \ell(Q)^{2-d}\,.
\end{equation}
Consider measure $\mu:= \sum_{m=1}^M \ell(Q_m)^{-d}\cdot m_2| Q_m$ and the kernel $t(x,y):= (x-y)^{-2}, x, y\in  \cup_{m=1}^M Q_m$. 
Then the operator with the kernel $t(x,y)$ is bounded from $L^2(\mu)$  to $L^2(\mu)$ in the sense that ($X:=\cup_{m=1}^M Q_m$)
\begin{equation}
\label{mainineq}
\int_X\bigg | \int_X \frac{f(y)}{(x-y)^2 }\,dm_2(y)\bigg|\,d\mu(x) \le C_1\, \int_X |f(x)|^2\, d\mu(x)\,.
\end{equation}
\end{theorem}

\noindent{\bf Remarks.} 1) The proof of this result in \cite{LSUT} is not too easy, and its mechanism is not obvious. Our goal in the present note is to clarify this mechanism.

2) We can write $d\mu= w\, dm_2$ with the obvious $w(x): =\sum_{m=1}^M  \ell(Q_m)^{-d}\cdot  \chi_{Q_m}$. It is very easy to check that
\begin{equation}
\label{A2}
\frac1{|B|}\int_{B\cap X} w\,dm_2 \cdot \frac1{|B|}\int_{B\cap X} w^{-1}\, d m_2 \le C_3
\end{equation}
independently of the disc $B$ on the plane. This, of course, suggests that the theorem above is a particular case of a classical Hunt--Muckenhoupt--Wheeden weighted $A_2$ result.  We warn the reader: this is {\bf not so at all}. A very ``tiny" difference is in integration over $B\cap X$ in \eqref{A2}.
It cannot be replaced by the integration over $B$, which practically always equal to infinity for $w^{-1}$. This tiny difference is crucial. It gives the idea that we need to stick to $X$, namely,  it seems at the first glance that we need the  Hunt--Muckenhoupt--Wheeden weighted $A_2$ result on {\it metric space} $X$.  However, at this moment there is no necessary and sufficient condition of the weighted type for singular integrals operators (SIO) on metric spaces. This problem is very close to a notoriously difficult two-weight problem for singular integrals. The difficulty is that we come into the realm of {\it non-homogeneous} Harmonic Analysis: meaning that neither $m_2|X$, nor $\mu|X$ is a doubling measure. There are partial $A_2$ type results in such situations, but only partial, see \cite{NRV} for example.

3) This explains the subtlety of the proof in \cite{LSUT}.

4) However, we show below that the point of view of non-homogeneous Harmonic Analysis on metric spaces is fruitful and allows us to simplify the proof of Theorem \ref{main}. But strangely enough, we need to look at it as {\bf unweighted} non-homogeneous $T1$ theorem (in fact, it turns out that the mechanism is even simpler, but for that the reader should look at Section \ref{simple}).

\section{A simple proof of Theorem \ref{main}. The weighted estimate of Ahlfors--Beurling transform = unweighted estimate of a certain non-symmetric Calder\'on--Zygmund operator on a metric space.}
\label{simple}
 
 Let 
 $$
 \mu=\sum_i \ell(Q_i)^{-d} m_2|Q_i\,.
 $$

Let us consider
$$
K(x,y):=\begin{cases} 0, \,\text{if}\,\, x, y \in \text{the same}\,\, Q_i, i=1,\dots, M\,;\\
\frac{\ell(Q_i)^d}{(x-y)^2}\,,\,\text{if}\,\, y\in  Q_i, x\in Q_j, i\neq j\end{cases}
$$
It is obvious that the boundedness of this operator in $L^2(\mu)$ is exactly equivalent to proving Theorem \ref{main}.

We notice that it is enough to prove the boundedness for the formal adjoint operator $T'_{\mu}$, whose kernel is 
$$
K'(x,y):=\begin{cases} 0, \,\text{if}\,\, x, y \in \text{the same}\,\, Q_i, i=1,\dots, M\,;\\
\frac{\ell(Q_j)^d}{(x-y)^2}\,,\,\text{if}\,\, y\in  Q_i, x\in Q_j, i\neq j\,.\end{cases}
$$

Consider any metric space  with any measure (we do not need even $(X,d)$ to be a geometrically doubling metric space), and let
$$
M_{\mu,3} f(x):=\sup_{R>0} \frac1{\mu(B(x, 3R)} \int_{B(x,R)} |f(y)|\,d\mu(y)\,.
$$
This maximal operator was widely used in \cite{NTV}. It is an immediate consequence of Vitali's covering lemma that it is bounded in $L^2(\mu)$, no matter what is $\mu$.

Now we have Theorem \ref{main} as an immediate corollary of

\begin{lemma}
\label{bd}
The operator $f\rightarrow \int_X K'(x,y) f(y)\, d\mu(y)$ is pointwise majorized by $C\,M_{\mu,3} f(x)$.
\end{lemma}

\begin{proof} Let $f\ge 0$. Fix $Q_j$, its center $c_j$, $x\in Q_j$, and let us estimate
$$
|T'_{\mu}f(x)| \le \int_X |K'(x,y)| f(y)\,d\mu(y)\,.
$$
Consider the centers $c_i$ of squares $Q_i$ such that
$c_i\in 2^{a+1} Q_j\setminus 2^a Q_j$. We call this family of centers (and their squares too) family $F_a$. Notice that $F_a$ can be non-empty only if  $a\ge 2$ (see \eqref{D}).

If $x\in Q_j, y\in Q_i, Q_i\in F_a$ then
$$
|K'(x,y)|\le \frac{\ell_j^d}{|c_j-c_i|^2}\,,\text{where}\,\, \ell_j:=\ell(Q_j)\,.
$$
Therefore,
$$
|T'_{\mu} f(x)|\le \sum_{a=2}^{\infty} \bigg( \sum_{c_i\in F_a} \frac1{|c_i-c_j|^2} \int_{Q_i} f\,d\mu\bigg) \ell_j^d\le
$$
$$
\sum_{a=2}^{\infty}2^{-ad} \frac{(2^a \ell_j)^d}{(2^a\ell_j)^2} \int_{\cup_{Q_i\in F_a}Q_i} f\,d\mu \le
$$
$$
\sum_{a=2}^{\infty}2^{-ad} \frac{1}{(2^a\ell_j)^{2-d}}\int_{B(x, 8\cdot 2^{a+1}\ell_j)} f\,d\mu=: I\,.
$$
The last inequality holds because each square from the family $F_a$ lie in the disc centered at $x\in Q_j$ of radius $R_a:= 8\cdot 2^{a+1}\ell_j$.
This is again obvious from \eqref{D}. We continue:
$$
I\le C_1 \sum_{a=2}^{\infty} 2^{-ad}  \frac{1}{(24\cdot 2^{a+1}\ell_j)^{2-d}}\int_{B(x, 8\cdot 2^{a+1}\ell_j)} f\,d\mu \le
$$
$$
C_2 \sum_{a=2}^{\infty} 2^{-ad}  \frac{1}{\mu(B(x, 3R_a))}\int_{B(x, R_a)} f\,d\mu\le C_3 \, M_{\mu,3}f(x)\,.
$$

The last inequality is because
\begin{equation}
\label{meest}
\mu(B(x, 24\cdot 2^{a+1}\ell_j)) \le C_4\, (24\cdot 2^{a+1}\ell_j)^{2-d}\,.
\end{equation}
In its turn this is easy: the disc $B(x, 24\cdot 2^{a+1}\ell_j)$ is covered by a fixed number of dyadic squares of comparable size. Apply the packing condition \eqref{P} to each of these dyadic squares, and see  \eqref{meest} immediately. Theorem \ref{main} is proved.
\end{proof}

\section{$T1$ theorem for non-homogeneous metric measure spaces}
\label{two}

The current section is not needed for the proof of Theorem \ref{main} because the proof has been already given in the previous section. 

Still we decided to include it here to explain the connection to other related questions. Also this section serves as a conceptual explanation of what is going on in Theorem \ref{main}. For the experts we want to emphasize  a ceratin affinity to the problem considered in Volberg--Wick's paper \cite{VW}. But the situation in \cite{VW} is actually more singular so-to-speak.

Let us  work in this section only with so-called geometrically doubling metric spaces (GDMS) $(X,d)$ meaning that each ball of radius $r$ can be covered by at most fixed number of balls of radius $r/2$. This guarantees (by Konyagin--Volberg's theorem \cite{KV}) that $X$ is a metric  measure space with measure $\sigma$, $\supp\sigma =X$, such that $\sigma$ is a doubling measure ($(X,d, \sigma)$ is sometimes called ``homogeneous metric-measure space"). However,  we need a non-homogeneous metric-measure space with the same underlying $(X,d)$: namely, $(X,d, \mu)$, where $\mu$ is {\bf not necessarily} doubling. 

Given such $(X,d, \mu)$ with an extra condition
\begin{equation}
\label{ms}
\mu(B(x,r) \le r^s
\end{equation}
we can consider singular kernel of singularity $s$ of Calder\'on--Zygmund type on $X$. It is $K(x,y), x, y\in X$ such that

I) $ |K(x,y)| \le d(x,y)^{-s}$

II) $ |K(x, y)- K(x', y)| \le \frac{d(x,x')^{\e}}{d(x,y)^{s+\e}}$ for some $\e>0$ and all $x, x', y$ such that $d(x,x')\le \frac12 d(x,y)$.

III) $ |K(x, y)- K(x, y')| \le \frac{d(x,x')^{\e}}{d(x,y)^{s+\e}}$ for the same $\e>0$ and all $y, y', x$ such that $d(y,y')\le \frac12 d(x,y)$.

We can consider the (formal) operator $T_\mu$.
$$
f\in L^2(\mu) \rightarrow \int_X K(x, y) f(y) d\mu(y)\,.
$$
Given all this, we have the following theorem called non-homogeneous $T1$ theorem. It was proved by Nazarov--Treil--Volberg if $X$ is a Euclidean spaces, but the same proof can be adapted for GDMS (and actually this has been done by Hyt\"onen--Martikainen \cite{HyMa}). Notice that the kernel actually can be allowed to be considerably worse (bigger) depending on $\mu$ than we list in I), II), III). See, for example, \cite{HyMa}, \cite{NRV}.  But we do not need this here. 

In the next result we assume that formal operator $T_{\mu}$ and its formal adjoint $T'_{\mu}$ can be correctly defined on characteristic functions of balls. 

\begin{theorem}
\label{T1}
Operator $T_{\mu}$ with kernel $K$ is bounded (can be extended to be a bounded operator from finite linear combinations of characteristic functions of balls) in $L^2(\mu)$ if and only if

(i) $\|T_{\mu}\chi_B\|_{\mu}^2 \le C\,\mu(B)\,,$

(ii) $\|T_{\mu}\chi_B\|_{\mu}^2 \le C\,\mu(B)\,,$

\noindent where  $T'_{\mu}$ is a (formal) operator with kernel $K(y,x)$.

\end{theorem}

Let us look at Lacey--Sawyer--Uriarte-Tuero's theorem \ref{main} from the point of view of this theorem.  The metric space is $X=\cup_{m=1}^M Q_m$, the metric $d$ is the usual Euclidean metric, $\mu= \sum_m \ell(Q)^{-m} \cdot m_2|Q_m$. What is $K(x,y)$? Of course it is not $t(x,y)= (x-y)^{-2}$. 

Let us consider
$$
K(x,y):=\begin{cases} 0, \,\text{if}\,\, x, y \in \text{the same}\,\, Q_i, i=1,\dots, M\,;\\
\frac{\ell(Q_i)^d}{(x-y)^2}\,,\,\text{if}\,\, y\in  Q_i, x\in Q_j, i\neq j\end{cases}
$$
It is really a simple calculation to see 

\begin{lemma}
\label{cz}
Let $\tau\in (0,1)$ be an arbitrary number, $s:=2-d$, $\e:=\min (1, \tau d)$. Then thus defined $K$ is a Calder\'on--Zygmund kernel
on $X$ with singularity $s$ and Calder\'on--Zygmund parameter $\e$.
\end{lemma}

\begin{proof}
It is trivial to see I) and II) by using disjointness condition \eqref{D}. To check III) one consider cases:

III.1) $y,y'$ are in the same square, $x$ is in a different square. This is easy again by \eqref{D}. If $x$ joins $y, y'$ in their square there is nothing to prove as $K$ becomes zero.

III.2) $x, y$ are in the same square but $y'$ is in a different one. Then \eqref{D} and $|y-y'|\le \frac12 |x-y|$ would imply that such situation is impossible. If  $x, y'$ are in the same square but $y$ is in a different one. Then  again \eqref{D} and $|y-y'|\le \frac12 |x-y|$ would imply that such situation is impossible.

III.3) $x, y, y'$ are all in different squares. Let $y\in Q_{i}$, $y'\in Q_{i'}$. As $|y-y'|\le \frac12 |x-y|$ and $\ell(Q_i)\le |y-y'|, \,\ell(Q_{i'} )\le |y-y'|$ (by \eqref{D}), 
we should think that squares $Q_i, Q_{i'}$ are small with respect to $|x-y|$. Here we do not estimate the difference, we estimate $|K(x, y)$ and $|K(x, y')|$ separately:
$$
|K(x,y)-K(x, y')| \le |K(x,y)|+ |K(x, y')| \le \frac{\ell(Q_i)^d +\ell(Q_{i'}^d}{|x-y|^2} \le $$
$$
C\frac{ |y-y'|^{\tau d}|y-y'|^{(1-\tau)d}}{|x-y|^2}
 \le 
C\, \frac{ |y-y'|^{\tau d}}{|x-y|^{2-d+\tau  d}}=C\, \frac{ |y-y'|^{\tau d}}{|x-y|^{s+\tau  d}}
$$

\end{proof}

\noindent{\bf Remark.}  The  integral operator with kernel $\int \frac1{(x-y)^2}...dm_2(y)$  from Theorem \ref{main}
 is exactly the sum of $\int K(x,y)...d\mu(y)$ (we call attention of the reader to the change of measure!) plus the local operator $\int t_0(x,y) ... d m_2(y)$, where 
$$
t_0(x,y) :=\sum_{m=1}^M\frac{\chi_{Q_m}(x)\chi_{Q_m}(y)}{(x-y)^2}\,.
$$
The boundedness of $t_0$ in each $L^2(Q_m, \ell(Q_m)^{-d}\,dm_2)$ (and, thus, in the direct sum of these spaces, which is precisely $L^2(X, \mu)$) is obvious as the constant $\ell(Q_m)^{-d}$ just cancels out in norm estimate in the left and right hand sides, and operator with kernel $\frac1{(x-y)^2}$ is a classical (called Ahlfors--Beurling) operator, which is an isometry in $L^2(m_2)$.

So, in principle, we could have checked assumptions (i), (ii) of our $T1$ theorem \ref{T1} and then Theorem \ref{main} would follow.

However, there is an even simpler way to see that the integral operator 
$$
\int K(x,y)...d\mu(y)
$$
is bounded in $L^2(\mu)$.

\

\bibliographystyle{amsplain}

\end{document}